\documentclass[12pt]{amsart}
\usepackage{amsmath, amsthm}
\usepackage[english]{babel}
\usepackage[T1]{fontenc}
\usepackage[latin1]{inputenc}
\usepackage{amssymb}
\usepackage{amscd}
\usepackage{latexsym}
\usepackage[bookmarksnumbered,plainpages,hypertex]{hyperref}
\usepackage{graphicx}
\usepackage{mathrsfs} 

\newtheoremstyle{theorem}
  {12pt}          
  {12pt}  
  {\sl}  
  {\parindent}     
  {\bf}  
  {. }    
  { }    
  {}     
\theoremstyle{theorem}
\newtheorem{theorem}{Theorem}
\newtheorem{corollary}[theorem]{Corollary}
\newtheorem{remark}[theorem]{Remark}
\newtheorem{proposition}[theorem]{Proposition}
\newtheorem{lemma}[theorem]{Lemma}

\newtheorem{definition}[theorem]{Definition}

\linethickness{0.5pt}

\newcommand{\ic}{\ensuremath{\mathcal{I}}}

\newcommand{\oc}{\ensuremath{\mathcal{O}}}

\newcommand{\ec}{\ensuremath{\mathcal{E}}}

\newcommand{\cc}{\ensuremath{\mathcal{C}}}

\newcommand{\Pt}{\mathbb{P}^3}
\newcommand{\Ptw}{\mathbb{P}^2}

\newcommand{\Pn}{\mathbb{P}^n}

\newcommand{\bP}{\mathbb{P}}

\newcommand{\aG}{\alpha}
\newcommand{\bG}{\beta}

\newcommand{\fG}{\varphi}

\newcommand{\sG}{\sigma}

\newcommand{\lag}{\langle}
\newcommand{\rag}{\rangle}

\newcommand{\bds}{\begin{displaystyle}}
\newcommand{\eds}{\end{displaystyle}}

\title[a.C.M. space curves reloaded.]{Arithmetically Cohen-Macaulay space curves reloaded.}

\author{Ph. Ellia}
\address{Dipartimento di Matematica, 35 via Machiavelli, 44100 Ferrara}
\email{phe@unife.it}


\begin{document}
\maketitle

\thispagestyle{empty}

\begin{abstract} We characterize the minimal free resolutions of zero-dimensional subschemes in the plane with non connected numerical character (i.e. with Hilbert function of non strictly decreasing type). This is then used to slightly generalize a result of Sauer about the smoothability of arithmetically Cohen-Macaulay space curves. Some complements are also given.
\end{abstract}


\section*{Introduction} We work over an algebraically closed field of characteristic zero. A curve $C \subset \Pt$ is a one-dimensional, equidimensional, closed subscheme which is locally Cohen-Macaulay (without embedded points). A curve $C \subset \Pt$ is arithmetically Cohen-Macaulay (a.C.M.) if and only if $H^1_*(\ic _C)=0$; if moreover $C$ is smooth (hence irreducible), then $C$ is said to be \emph{projectively normal} (just p.n. in the sequel). Finally a curve, $C$, is \emph{smoothable} if there exists a flat family $\cc \subset \Pt _T$ such that $\cc _{t_0}=C$ and $\cc _t$ is smooth for general $t \in T$. 

 In \cite{S} Tim Sauer proved the following nice result:
{\it Let $C \subset \Pt$ be an a.C.M. curve with minimal free resolution
$$0 \to \bigoplus _{j=1}^k \oc (-b_j) \to \bigoplus _{i=1}^{k+1}\oc (-a_i) \to \ic _C \to 0$$
$b_k \geq ... \geq b_1$, $a_{k+1}\geq ... \geq a_1$. Then $C$ is smoothable if and only if: $b_n \geq a_{n+2}$ for $1 \leq n \leq k-1$.}

The purpose of this note is just to relate this result to earlier work on arithmetically Cohen-Macaulay space curves to get a unified and general perspective. In the meantime this will put in evidence some results, which are maybe well-known to experts, but for which it is difficult to give a reference.\\
First we show (cf Thm. \ref{thm1}) that in the case of zero-dimensional subschemes of $\Ptw$ Sauer's condition on the minimal free resolution characterizes zero-dimensional subschemes with a \emph{connected} numerical character (i.e. with an Hilbert function of strictly decreasing type). Next we show how Sauer's result is related with earlier work of Ellingsrud (\cite{E}) and Gruson-Peskine (\cite{GP}). We thus obtain (Thm. \ref{thm2}) a more general smoothability statement (with a different proof). Some complements ((4) of Thm. \ref{thm2}, Cor. \ref{leCor}, Prop. \ref{EffCara}) on a.C.M. curves are given.

\section{Points in the plane.}

If $Z \subset \Ptw$ is a zero-dimensional subscheme we denote by $s(Z)$ (or just $s$ if no confusion can arise) the minimal degree of a curve containing $Z$.
\par
There are several ways to encode the Hilbert function of a zero-dimensional subscheme of $\Ptw$, here we will use the \emph{numerical character}, $\chi (Z)$ (cf \cite{GP}, \cite{EP}). We recall that $\chi (Z)$ is a sequence of integers $(n_0,...,n_{s-1})$ such that:
\begin{enumerate}
\item $n_0 \geq n_1 \geq ... \geq n_{s-1}\geq s$ ($s = s(Z)$)
\item $h^1(\ic _{Z}(n) = \bds  \sum _{i=0}^{s-1}([n_i-n-1]_+ -[i-n-1]_+ ) \eds$
\item In particular $\deg (Z )=\bds \sum _{i=0}^{s-1}(n_i-i).\eds$
\end{enumerate}
It is clear (cf (2)) that $\chi (Z)$ determines $H(Z,-)$ (and vice versa).

For those more comfortable with the Hilbert function and its first difference function, $\Delta (Z,i)=H(Z,i)-H(Z,i-1)$, we recall that $\Delta (i)=i+1$ for $i < s$ while $\Delta (i) = \# \{l \mid n_l \geq i+1\}$ for $i \geq s$. (This can be used as a definition of $\chi (Z)$.)\\
More generally we make the following:

\begin{definition}
\label{carac}
A numerical character $\chi$ of length $s$, degree $d$, is a sequence of integers $(n_0,...,n_{s-1})$ with $n_0\geq ...\geq n_{s-1}\geq s$ and with $\sum _{i=0}^{s-1}(n_i-i)=d$. Moreover we say that $\chi$ is \emph{connected} if $n_i \leq n_{i+1}+1$, $0 \leq i \leq s-2$.
\end{definition}

Observe that $\chi (Z)$ is connected if and only if $H(Z,-)$ is of strictly decreasing type. 

A zero-dimensional subscheme $Z \subset \Ptw$ is a.C.M., its graded ideal has a length one minimal resolution:
$$0 \to \bigoplus _{j=1}^k \oc (-b_j) \stackrel{\fG}{\to} \bigoplus _{i=1}^{k+1}\oc (-a_i) \to \ic _Z \to 0$$
The integers $b_j, a_i$ are the Betti numbers of $Z$. We will always assume that $b_k \geq ...\geq b_1$, $a_{k+1}\geq ...\geq a_1$. We will also say that $Z$ is of type $(\underline a, \underline b)$ ($\underline a = (a_1,...,a_{k+1})$, $\underline b=(b_1,...,b_k)$). The morphism $\fG$ is given by a $(k+1)\times k$ matrix of homogeneous polynomials $\fG _{ij}$ of degree $b_j-a_i$. By minimality if $b_j=a_i$, then $\fG _{ij}=0$. If $\Delta _i$ is the $k\times k$ minor obtained by suppressing the $i$-th row, then $(\Delta _1,...,\Delta _{k+1})$ is a set of minimal generators of $I(Z)$. Observe that $\deg (\Delta _i)=a_i$.

If we are given the Betti numbers, we can determine the Hilbert function, but the converse is not true: the Hilbert function determines the Betti numbers only \emph{up to repeated terms} (cf Lemma \ref{niaibj}).

\begin{remark}
\label{mfrCar} 
Here are some relations, which easily follow from the definitions, between the Betti numbers $(a_i),(b_j)$ and the numerical character $(n_i)$ of a zero-dimensional subscheme.\\
(1) $a_1=s$ and $a_2=n_{s-1}$\\
(2) $b_k = n_0+1$ and $\#\{j \mid b_j=b_k\} = \#\{i \mid n_i=n_0\} = h^1(\ic _Z(b_k-3))$\\
(3) $a_{k+1} \leq n_0$.
\end{remark}

We will need also the following:

\begin{lemma}
\label{niaibj}
Let $Z \subset \Ptw$ be a zero-dimensional subscheme with minimal free resolution
$$0 \to \bigoplus \bG _i\oc (-i) \to \bigoplus \aG _i\oc (-i) \to \ic _Z \to 0$$
and with $\chi (Z)=(n_0, ..., n_{s-1})$. Let $c(n)=\#\{i \mid n_i=n\}$. Then:
\begin{enumerate}
\item $\aG _s = c(s)+1$
\item For $n > s$: $\bG _n = \aG _n -c(n)+c(n-1)$.
\end{enumerate}
\end{lemma}

\begin{proof} (1) follows from the definitions, (2) is proved by induction. First observe that $h^0(\ic _Z(n+1))-h^0(\ic _Z(n))=n+2-s+c(n+1)+...+c(s)\,\,(*)$. Then from the minimal free resolution: $h^0(\ic _Z(n)) = \sum _{i\leq n}(\aG _i-\bG _i)h^0(\oc (n-i)$, and similarly for $h^0(\ic _Z(n+1))$. It follows that: $h^0(\ic _Z(n+1))-h^0(\ic _Z(n))=  \sum _{i=s}^{n+1}(\aG _i-\bG _i)(n+2-i)$. We conclude using $(*)$ and the inductive assumption. 
\end{proof}

Concerning the connectedness of the numerical character we have:

\begin{proposition}
\label{propEP}
Let $Z \subset \Ptw$ be a zero-dimensional subscheme with $\chi (Z) = (n_0,..,n_{s-1})$. If $n_{t-1}>n_t+1$ for some $t$, then there exists a degree $t$ curve, $T$, such that $X = Z\cap T$ has $\chi (X)=(n_0, ..., n_{t-1})$. Moreover if $R$ is the residual scheme of $Z$ with respect to $T$, then $\chi (R)=(m_0, ..., m_{s-1-t})$ where $m_i=n_{t+i}-t$.
\end{proposition}

\begin{proof} See \cite{EP}, Proposition. \end{proof}

\begin{corollary}
\label{sInt}
Let $Z \subset \Ptw$ be a zero-dimensional subscheme with $\chi (Z) = (n_0,..,n_{s-1})$. If $Z$ lies on an integral curve of degree $s$, then $\chi (Z)$ is connected.
\end{corollary}

\begin{proof} Assume $n_{t-1}>n_t+1$. We have $n_{t-1}>n_t+1\geq s+1>t$. We apply \ref{propEP}. Every curve, $S$, of degree $s$ containing $Z$ contains $X =Z\cap T$. The first generator of $I(X)$ is $T$ of degree $t$, the next one has degree $n_{t-1}$ (Remark \ref{mfrCar}). Since $n_{t-1}>s$, $T$ divides $S$.
\end{proof}

The following is well known, we include it for completeness and also to give some perspective to the next result.

\begin{proposition}
\label{Exi}
Let $b_k \geq ... \geq b_1 > 1$ and $a_{k+1}\geq ...\geq a_1 > 0$ be integers. There exists a zero-dimensional closed subscheme $X \subset \Ptw$, with minimal free resolution:
$$0 \to \bigoplus _{j=1}^k \oc (-b_j) \stackrel{\fG}{\to} \bigoplus _{i=1}^{k+1} \oc (-a_i) \to \ic _X \to 0\,\,\,\,(*)$$
if and only if:\\
(a) $\fG$ is minimal (i.e. if $b_j=a_i$, then $\fG _{ij}=0$).\\
(b) $\sum b_j = \sum a_i$\\
(c) $b_1 > a_2$, $b_2 > a_3$, ..., $b_k > a_{k+1}$.\\
\end{proposition}

\begin{proof}
Assume $X$ has a minimal free resolution like $(*)$. Condition (a) and (b) are clearly satisfied. The morphism $\fG$ is given by a matrix, with $k+1$ rows and $k$ columns, $(\fG _{ij})$, where $\fG _{ij}$ is an homogeneous polynomial of degree $b_j-a_i$.\\
Take $m$, $1 \leq m\leq k$. Let $q =max\{t \mid (\fG _{t,1},....,\fG _{t,m})\neq 0\}$. Let $L_1:=\bigoplus _{j=1}^k \oc (-b_j)$, $L_0=\bigoplus _{i=1}^{k+1}\oc (-a_i)$, $N_1:=\bigoplus _{j=1}^m \oc (-b_j)$, $N_0:=\bigoplus _{i=1}^q \oc (-a_i)$. Then $\fG$ induces an injective morphism: $N_1 \to N_0$ and we have a commutative diagram:
$$\begin{array}{ccccccccc}

  	& 		  & 0   						&    	&   0							&     		&     		&  		&  \\ 
  	& 		  & \downarrow   		&    	&\downarrow   		&     		&     		&  		&  \\ 
0 	& \to		& N_1		& \to 	& N_0		& \to 						& \ec 		& \to 		& 0 \\ 
  	& 		  & \downarrow   		&    	&\downarrow   		&     		&\downarrow f     		&  		&  \\ 
0 	& \to		& L_1		& \to 	& L_0		& \to 		& \ic _X 		& \to 		& 0 \\
  	& 		  & \downarrow   		&    	&\downarrow   		&     		&     		&  		&  \\ 
  	&   		& \overline L_1		& \to 	& \overline L_0		&       		&     		&     		&  \\ 
  	& 		  & \downarrow   		&    	&\downarrow   		&     		&     		&  		&  \\ 
  	& 		  & 0   						&    	&   0							&     		&     		&  		&  \end{array}$$

By the snake-lemma we have an exact sequence:
$$0 \to Ker(f) \to \overline L_1 \to \overline L_0 \to Coker(f) \to 0$$
If $\ec$ has rank zero, then $rk(Ker(f))=0$, since $\overline L_1$ is torsion free, this implies $Ker(f)=0$, this in turn imply $\ec =0$ ($\ic _X$ is torsion free). Since $N_1$ and $N_0$ can't be isomorphic (for sure $a_1 < b_1$), we conclude that $rk(\ec )\geq 1$. It follows that $q-m\geq 1$, hence $a_{m+1}\leq a_q$. By assumption there exists $i \leq m$ such that $\fG _{qi}\neq 0$, hence $a_q < b_i$. We conclude that: $a_{m+1}\leq a_q < b_i \leq b_m$, thus: $a_{m+1}<b_m$.
\par
($\Leftarrow$) Define the matrix $\fG = (\fG _{ij})$ by $\fG _{j+1,j}= x_2^{b_j-a_{j+1}}$, $\fG _{j,j}=x_1^{b_j-a_j}$, $\fG _{j-1,j}=x_0^{b_j-a_{j-1}}$, $1 \leq j \leq k$, $\fG _{ij}=0$ otherwise. Since $b_j > a_{j+1}\geq a_j \geq a_{j-1}$, it is easy to see that outside of the point $p=(1:0:0)$, the matrix $\fG$ has rank $k$. It follows (using (1)), that $Coker(\fG )\simeq \ic _X$ for some $X$ zero-dimensional subscheme (with support $p$).\\  
\end{proof}

Now we come to our main result.

\begin{theorem}
\label{thm1}
Let $Z \subset \Ptw$ be a zero-dimensional subscheme with minimal free resolution
$$0 \to \bigoplus _{j=1}^k \oc (-b_j) \to \bigoplus _{i=1}^{k+1}\oc (-a_i) \to \ic _Z \to 0$$
where $b_k \geq ...\geq b_1$, $a_{k+1}\geq ...\geq a_1$. The following are equivalent:
\begin{enumerate}
\item $\chi (Z)$ is not connected
\item $b_p \leq a_{p+2}$ for some $p$, $1 \leq p \leq k-1$.
\end{enumerate}
\end{theorem}

\begin{proof} (1) $\Rightarrow$ (2): Let $\chi (Z)=(n_0,...,n_{s-1})$ and assume $n_{t-1}>n_t+1$. Then (cf \ref{propEP}) there exists a degree $t$ curve, $T$, and $X \subset Z$ such that $X = Z \cap T$ and $\chi (X)=(n_0,...,n_{t-1})$. Moreover if $R$ is the residual scheme of $Z$ with respect to $T$, then $\chi (R)=(m_0,...,m_{s-1-t})$, with $m_i=n_{t+i}-t$. In particular we have an exact sequence:
$$0 \to \ic _R(-t) \stackrel{.T}{\to} \ic _Z \to \ic _{X,T} \to 0\,\,\,(+)$$
If $I(R)=(f_1,...,f_r)$, then $g_1=Tf_1, ..., g_r=Tf_r$ are minimal generators of $I(Z)$ of lowest degrees. Indeed $\deg (f_i)\leq m_0$ and $h^0(\ic _{X,T}(k))=0$ if $k < n_{t-1}$ (Remark \ref{mfrCar}). Let $(\underline \aG, \underline \bG )$ be the Betti numbers of $R$ ($\aG _i=\deg (f_i), 1 \leq i \leq r$). Then the first Betti numbers of $Z$ are $a_1=\aG _1 +t, ...,a_r=\aG _r+t$; $b_1 = \bG _1+t, ..., b_{r-1}=\bG _{r-1}+t$. Moreover since $\bG _{r-1}=m_0+1$ (cf Remark \ref{mfrCar}), $b_{r-1}=n_t+1$. The next generator of $I(Z)$, of degree $a_r$, comes from $H^0_*(\ic _{X,T})$ (consider exact sequence $(+)$). It has degree $n_{t-1}$ (Remark \ref{mfrCar} again). It follows that $n_t+1 =b_{r-1}<a_{r+1}=n_{t-1}$.

(2) $\Rightarrow$ (1): assume $b_p < a_{p+2}$, then, $A$, the matrix of $\fG$, has the following shape: $A = \left (  \begin{array}{cc}
A_1 & A_2 \\ 
0& A_3
\end{array} \right ) $ where $A_1$ is a $(p+1)\times p$ matrix and where $A_3$ is a square $k-p$ matrix. If $T:= \det A_3$, then the generators of lowest degrees (i.e. the minors of $A$ obtained by deleting the $i$-th row, $1 \leq i \leq p+1$) are of the form $g_1 =Tf_1, ..., g_{p+1}=Tf_{p+1}$. Observe that $\deg g_i=a_i$.

The $(p+1)\times p$ matrix $A_1$ defines a zero-dimensional subscheme $R$ with $I(R)=(f_1, ..., f_{p+1})$. If $N_1 = \bigoplus _{j=1}^p \oc (-b_j)$ and $N_0=\bigoplus _{i=1}^{p+1}\oc (-a_i)$, then:\\
$$0 \to N_1(-t) \to N_0(-t) \to \ic _R \to 0$$
is the minimal free  resolution of $R$ and we have $0 \to \ic _R(-t) \stackrel{.T}{\to} \ic _Z$ whose cokernel has support on $T$.

Let $\chi (R)=(m_0, ..., m_{s-1-t})$ be the numerical character of $R$. If $\chi (Z)=(n_0, ..., n_{s-1})$ we claim that $n_{i+t}=m_i+t$, $0 \leq i \leq s-1-t$. This follows from Lemma \ref{niaibj}, indeed if $n = m_0+t$, then since $c(k)=\aG _k-\bG _k+c(k-1)$ (notations as in Lemma \ref{niaibj}), we see that for $k\leq n$, the $n_i's$ less than $n$ are determined by the generators and relations of degree $\leq n$. Since $\deg g_{p+2}=a_{p+2}$ and $a_{p+2} > b_p = m_0+t+1$ (see Remark \ref{mfrCar}), the $n_i's$ less than $n$ are determined by the generators $g_1=Tf_1, ..., g_{p+1}=Tf_{p+1}$ and by their relations: this is encoded (modulo the factor $T$) in the minimal free resolution of $I(R)=(f_1, ..., f_{p+1})$ and in $\chi (R)$.

It remains to show that $n_{t-1}>n_t+1 =m_0+t+1$. If $n =m_0+t+1$ we have: $\bG _n =\aG _n-c(n)+c(n-1)$. By Remark \ref{mfrCar} (2): $\bG _n=c(n-1)$. Since $\aG _n =0$ (because $a_{p+2}>b_p$), we get $c(n)=0$, i.e. no $n_i$ is equal to $n_{t-1}+1$, hence $n_{t-1}>n_t+1$.    
\end{proof}

\section{Arithmetically Cohen-Macaulay space curves.}

If $C \subset \Pt$ is any curve, the numerical character of $C$ is the character of its general plane section.

It follows from the works of Ellingsrud (\cite{E}) and Gruson-Peskine (\cite{GP}) that a.C.M. curves are classified by their numerical character. 

\begin{theorem}[Ellingsrud-Gruson-Peskine]\quad
\label{EGP}
\begin{enumerate}
\item An a.C.M. curve $C \subset \Pt$ corresponds to a smooth point of $Hilb(\Pt )$
\item Two a.C.M. curves $C,X \subset \Pt$ are in the same irreducible component of the Hilbert scheme if and only $\chi (C)=\chi (X)$
\item The numerical character of an integral curve is connected. Moreover every \emph{connected} numerical character is realized by a projectively normal curve
\item Let $H_{\chi}$ denote the irreducible component of $Hilb(\Pt )$ parametrizing a.C.M. curves with numerical character $\chi$. The general curve of $H_{\chi}$ is smooth (hence irreducible) if and only if $\chi$ is connected.
\end{enumerate}
\end{theorem}

Now we have:

\begin{theorem}
\label{thm2}
Let $X \subset \Pt$ be an a.C.M. curve with numerical character $\chi = (n_0, ..., n_{s-1})$ and minimal free resolution
$$0 \to \bigoplus _{j=1}^k \oc (-b_j) \to \bigoplus _{i=1}^{k+1}\oc (-a_i) \to \ic _X \to 0\,\,\,(*)$$
The following are equivalent:
\begin{enumerate}
\item $\chi$ is connected
\item $b_n \geq a_{n+2}$, for all $n$, $1 \leq n \leq k-1$
\item $X$ is smoothable
\item  The general curve of $H_{\chi}$ lies on a smooth surface of degree $s$
\item The general curve of $H_{\chi}$ lies on an integral surface of degree $s$.
\end{enumerate}
\end{theorem}

\begin{proof}
(1) $\Leftrightarrow$ (2) this follows from Theorem \ref{thm1} since, for $H$ a general plane, the minimal free resolution of $X\cap H$ is the restriction of $(*)$ to $H$ and $\chi (X\cap H)=\chi$.\\
(1) $\Leftrightarrow$ (3): this is  (4) of Theorem \ref{EGP}.\\
(1) $\Rightarrow$ (4): we show that for every connected character $\chi$ of length $s$ there exists a p.n. curve $C \subset \Pt$ with $\chi (C) = \chi$ and lying on a smooth surface of degree $s$. For this we take over the proof of Thm. 2.5 in \cite{GP}. The argument is by induction. Assume $C$ is a p.n. curve with $\chi (C) =(n_0,..., n_{s-1})$ and that $C$ lies on, $S$, a smooth surface of degree $s$. (If $s=1$ this is clearly satisfied). Following \cite{GP} we show the existence of a p.n. curve $C_1$ (resp. $C_2$) with $\chi (C_1)=(n_0 + 2, n_0 + 1, n_1 + 1,...,n_{s-1} + 1)$ (resp. $\chi (C_2)= (n_0 + 1, n_0 + 1, n_1 + 1,...,n_{s-1} + 1)$), lying on a smooth surface of degree $s+1$. Moreover, as in \cite{GP}, the following condition is also part of the induction:
$$\omega _C(-e(C))\,\,\, has\,\, a\,\, section\,\, \sG, with\,\, smooth\,\, zero-locus\,\,(\sG)_0\,\,(+)$$
The curve $C_1$ (resp. $C_2$) is constructed as follows: there exists a smooth surface, $F$, of degree $n_0+2$ (resp. $n_0+1$) containing $C$ (observe that $n_0=e+3$). If $L =\oc _F(C)$, then $C_1$ (resp. $C_2$) is a general section of $L(1)$. This means that $C_i$ is obtained by double linkage from $C$: $U:=F \cap G_a =X \cup C$ and $F \cap T_{a+1} =X \cup C_i$. If we take $G_a =S$, it is enough to show that $X$ lies on a smooth surface of degree $s+1$, $T$. The exact sequence of liaison:
$$0 \to \ic _U(s+1) \to \ic _X(s+1) \stackrel{r}{\to} \omega _C(-e+i-1) \to 0$$
shows that $H^0(\ic _X(s+1))$ contains (at least) $V = \lag H_iS, T' \rag$, where $r(T')=(\sG )_0$ (see $(+)$). Hence $T'$ is not a multiple of $S$. The base locus of $V$ is $B =T' \cap S$. By Bertini's theorem the general surface in $H^0(\ic _X(s+1))$ is smooth out of $B$. Since $S$ is smooth, the general surface $T = HS+T' \in V$ is smooth with $r(T)=(\sG )_0$. For general $T \in V$ the linked curve $C_i$ will be smooth  and will satisfy $(+)$ ($\omega _{C_i}(-e-3+i)$ has a section with zero-locus $(\sG )_0$ cut out by $S$ residually to $X \cap C_i$). We conclude as in \cite{GP}.\\
(4) $\Rightarrow$ (5): clear.\\
(5) $\Rightarrow$ (1): by considering the general plane section $Z =X\cap H$, this follows from Corollary \ref{sInt}.  
\end{proof}

We notice the following handy criteria:

\begin{corollary}
\label{leCor}
Let $X \subset \Pt$ be an a.C.M. curve of degree $d$, lying on an integral surface of degree $s$. If $d > s(s-1)$, then $X$ is smoothable.
\end{corollary}

\begin{proof} By degree reason $s=s(X)$. We conclude with (5) of Theorem \ref{thm2}.
\end{proof}

Finally, to really complete the picture:

\begin{proposition}
\label{EffCara} Every numerical character is effective. More precisely let $\chi$ be a numerical character then $H_{\chi} \neq \emptyset$ and the general curve of $H_{\chi}$ is reduced. 
\end{proposition}

\begin{proof} Since the cone over any zero-dimensional subscheme $Z \subset \Ptw$ is an a.C.M. space curve, it is equivalent to show that any character is the character of a zero-dimensional subscheme. We proceed by induction on $s$ and prove that any character is the character of (the general plane section of) a reduced curve. If $\chi$ is connected we are done by Thm. \ref{EGP}. Otherwise let $t$ be the smallest integer such that: $n_{t-1}>n_t+1$. Then $\overline \chi =(n_0, ..., n_{t-1})$ is a character ($n_{t-1}\geq t$), of length $<s$, which is connected. By Thm. \ref{EGP} there exists a p.n. curve $C$ such that $\chi (C) =\overline \chi $. Now let $\chi ' = (m_0, ..., m_{s-1-t})$ where $m_i=n_{t+i}-t$. Then $\chi '$ is a character, of length $<s$. By inductive hypothesis there exists a reduced curve $\tilde R$ with $\chi ( \tilde R)=\chi '$. Observe that $C$ is contained in a unique integral surface of degree $t$, $\tilde T$. We may assume $\dim (\tilde R \cap \tilde T)=0$. Let $Y = C\cup \tilde R$. Taking a general plane section we are reduced to the following situation: $X =C\cap H$ has $\chi (X)=\overline \chi$ and lies on an integral curve of degree $t$: $T=\tilde T \cap H$; $R =\tilde R \cap H$ has $\chi (R)=\chi '$ and $R \cap T=\emptyset$. Both $R$ and $X$ are smooth. We want to prove that $Z =X\cup R$ has $\chi (Z)=\chi$. By construction we have an exact sequence:
$$0 \to \ic _R(-t) \to \ic _Z \to \ic _{X,T} \to 0$$
Since $h^0(\ic _{X,T}(n))=0$ if $n < n_{t-1}$ (cf Remark \ref{mfrCar}), we see that $s(Z)=s$ and, arguing as in the proof of Theorem \ref{thm1}, that $\chi (Z)=(a_0, ..., a_{t-1},n_t=m_0+t, ..., n_{s-1}=m_{s-1-t}+t)$. So far all the generators of $I(R)$ and the relations between them have been taken into account. The next generator of $I(Z)$ has degree $\geq n_{t-1}$ (in fact $=n_{t-1}$ because $h^1(\ic _R(n_{t-1}-t))=0$). It follows that $a_{t-1}>n_t+1$. We conclude with Prop. \ref{propEP}.
\end{proof}

\begin{remark} Proposition \ref{EffCara} is a complement to Thm. \ref{EGP} and determines all possible Hilbert functions of zero-dimensional subschemes of $\Ptw$. It also shows that any possible Hilbert function is realized by a smooth set of points. For points (in $\Pn$) this has been first proved in \cite{GMR}.
\end{remark}



\begin{thebibliography}{PrecisionACM}

\bibitem{EP} Ellia, Ph.-Peskine, Ch.: {\it Groupes de points de $\Ptw$: caractère et position uniforme}, Lecture Notes in Math., 1417, 111-116 (1990)
\bibitem{E} Ellingsrud, G.: {\it Sur le schéma de Hilbert des variétés de codimension 2 dans $\bP ^e$ à c\^{o}ne de Cohen-Macaulay}, Ann. scient. \'Ec. Norm. Sup., $4^e$ série, t. 8, 423-432 (1975)
\bibitem{GMR} Geramita, A.V.-Maroscia, P.-Roberts, L.G.:{\it The Hilbert function of a reduced $k$-algebra}, J. London Math. Soc. (2), 28, 443-452 (1983)
\bibitem{GP} Gruson, L.-Peskine, Ch.: {\it Genre des courbes de l'espace projectif}, Lecture Notes in Math., 687, 31-59 (1978) 
\bibitem{S} Sauer, T.: {\it Smoothing projectively Cohen-Macaulay space curves}, Math. Ann., 272, 83-90 (1985)



\end{thebibliography}
\end{document}